\documentclass[12pt,oneside,reqno]{amsart}
\usepackage{amssymb,latexsym}
\usepackage{amsfonts}
\usepackage{amsmath}
\usepackage[colorlinks,linkcolor=blue,anchorcolor=blue,citecolor=blue]{hyperref}
\usepackage{algorithm}
\usepackage{enumerate}
\usepackage{algpseudocode}
\usepackage{verbatim}
\usepackage{graphicx}
\usepackage{multirow}

\newfont{\aaa}{cmb10 at 19pt}
\newfont{\bbb}{cmb10 at 11pt}

\newtheorem{theorem}{Theorem}[section]

\theoremstyle{remark}

\numberwithin{equation}{section}

\begin{document}

\title[]{On multi-variable Zassenhaus formula}

\author{Linsong Wang}
\address{School of Mathematics, South China University of Technology, Guangzhou, Guangdong 510640, China}
\email{lslswls@qq.com}
\author{Yun Gao}
\address{Department of Mathematics and Statistics, York University, Toronto}
\email{ygao@yorku.ca}
\author{Naihuan Jing$^\dagger$}
\address{School of Mathematics, South China University of Technology, Guangzhou, Guangdong 510640, China}
\address{Department of Mathematics, North Carolina State University, Raleigh, NC 27695, USA}
\email{jing@math.ncsu.edu}


\begin{abstract}
In this paper, we give a recursive algorithm to compute the multivariable Zassenhaus formula
$$e^{X_{1}+X_{2}+\cdots +X_{n}}=e^{X_{1}}e^{X_{2}}\cdots e^{X_{n}}\prod_{k=2}^{\infty}e^{W_{k}}$$
and derive an effective recursion formula of $W_{k}$.
\end{abstract}

\thanks{{\scriptsize
\hskip -0.4 true cm
$\dagger$ Corresponding author: jing@math.ncsu.edu\\
MSC (2010): Primary: 16W25; Secondary: 22E05, 16S20\\
Keywords: Baker-Campbell-Hausdorff formula, Zassenhaus formula
}}

\maketitle

\section{Introduction}

The celebrated Baker-Campbell-Hausdorff (BCH) is a fundamental identity in Lie theory \cite{Baker,Campbell,Hausdorff} connecting Lie algebra with Lie group. 
The BCH says that for any linear operators $X, Y$ in a bounded Hilbert space one has the formula
\begin{align}
e^{X}e^{Y}=e^{X+Y+\sum_{k=2}^{\infty}Z_{k}(X,Y)},
\end{align}
where $\exp$ is defined in the usual sense and $Z_k(X, Y)$ is a degree $k$ homogeneous Lie polynomial in the noncommutative variables $X$ and $Y$. The first few terms are 
\begin{align*}
&Z_{2}=\frac{1}{2}[X, Y], \quad Z_{3}=\frac{1}{12}([X, [X, Y]]-[Y, [X, Y]]), \\
&Z_{4}=\frac{1}{24}[X, [Y, [Y, X]]].
\end{align*}
and the general expressions of $Z_k(X, Y)$ can be explicitly computed by combinatorial formulas.

The dual form of the BCH is the famous Zassenhaus formula which establishes that the exponential $e^{X+Y}$ can be uniquely decomposed as
\begin{align}
e^{X+Y}&=e^{X}e^{Y}\prod_{m=2}^{\infty}e^{W_{m}(X, Y)}\\
&=e^{X}e^{Y}e^{W_{2}(X, Y)}e^{W_{3}(X, Y)}\cdots e^{W_{k}(X, Y)}\cdots, \notag
\end{align}
where $W_{k}(X, Y)$ is a homogeneous Lie polynomial in $X$ and $Y$ of degree $k$ \cite{Magnus}. The first few terms are
\begin{align*}
&W_{2}=-\frac{1}{2}[X, Y], \quad W_{3}=\frac{1}{3}[Y, [X, Y]]+\frac{1}{6}[X, [X, Y]], \\
&W_{4}=-\frac{1}{24}[X, [X, [X, Y]]]-\frac{1}{8}([Y, [X, [X, Y]]]+[Y, [Y, [X, Y]]]).
\end{align*}
There are several methods to compute $W_{k}$ \cite{Suzuki,Weyrauch,Scholz,Ki}. In particular, a recursive algorithm has been proposed in \cite{Casas} to express directly $W_{k}$ with the minimum number of independent commutators required at each degree $k$.

Similar to the BCH formula, the Zassenhaus formula is useful in many different fields: $q$-analysis in quantum groups \cite{Quesne}, quantum nonlinear optics \cite{Quesada}, the Schr\"{o}dinger equation in the semiclassical regime \cite{Bader}, and splitting methods in numerical analysis \cite{Geiser}, etc.

We now consider the multivariate BCH and Zassenhaus formulas. It is easy to obtain the multivariable BCH formula by repeatedly using the usual BCH: 
\begin{align} \label{x1}
e^{X_{1}}e^{X_{2}}\cdots e^{X_{n}}=e^{X_{1}+X_{2}+\cdots +X_{n}+\sum_{m=2}^{\infty}Z_{m}(X_{1}, X_{2}, \cdots , X_{n})},
\end{align}
where $Z_m$ is a Lie polynomial in the $X_i$ of degree $m$.
On the hand, we also have the multivariable Zassenhaus formula
\begin{align} \label{x2}
e^{X_{1}+X_{2}+\cdots +X_{n}}=e^{X_{1}}e^{X_{2}}\cdots e^{X_{n}}\prod_{k=2}^{\infty}e^{W_{k}}
\end{align}
where the product is ordered and $W_{k}$ is a homogeneous Lie polynomial in the $X_{i}$ of degree $k$. However, it is more complicated to express $W_k$ in terms of $X_i's$.

The existence of the formula \eqref{x2} is a consequence of Eq. \eqref{x1}. In fact, it is clear that $e^{-X_{1}}e^{X_{1}+X_{2}+\cdots +X_{n}}=e^{X_{2}+X_{3}+\cdots +X_{n}+D}$, where $D$ involves Lie polynomials of degree $>1$. Then $e^{-X_{2}}e^{X_{2}+X_{3}+\cdots +X_{n}+D}=e^{X_{3}+X_{4}+\cdots +X_{n}+W_{2}'+D_{1}}$, where $D_{1}$ is an infinite Lie power series in the $X_i$ with minimum degree $>2$. Note that $W_{2}'\neq W_{2}$, we need to repeat the process $(n-1)$ times to determine $W_{2}$, i.e.
\begin{align*}
&\quad e^{-X_{n-1}}e^{X_{n-1}+X_{n}+W_{2}^{(n-3)}+W_{3}^{(n-4)}+\cdots +W_{n-2}'+D_{n-3}}\\
&=e^{X_{n}+W_{2}+W_{3}^{(n-3)}+\cdots +W_{n-2}''+W_{n-1}'+D_{n-2}}
\end{align*}
where $D_{n-2}$ involves Lie polynomials of degree $>(n-1)$. Finally, we can get the formula \eqref{x2} by repeating the process.

In this paper, we will give a new recursive algorithm to compute $W_{k}$ in \eqref{x2}. Our method is inspired by the recent algorithm in \cite{Casas}, and our formula is based on a new formula for $f_{1, k}$ using compositions of integers.

The paper is organized as follows. 
In Section 2, we give our recursive algorithm and a concrete procedure to compute $W_{k}, k=1, 2, 3, 4, 5$. In Section 3. we establish
a combinatorial formula of $f_{1, k}$ (see Theorem \ref{3.1}). We will show that our formula can give a slightly better recursion formula of $W_{k}$ when $k>5$ in Theorem \ref{3.2}.
Finally we use examples to show how the $f_{1, k}$ are used to derive Lie polynomial formulas of $W_k$ in terms of the
operators $X_1, \ldots, X_n$. The latter set of formulas are expected be useful in the quantum control problem.

\section{Multivariable Zassenhaus terms}

\subsection{A recurrence.}
For the operators $X_1, \ldots, X_n$ we consider the following function of $t$:
\begin{equation}
e^{t(X_{1}+X_{2}+\cdots+X_{n})}=e^{t X_{1}}e^{t X_{2}}\cdots e^{t X_{n}}e^{t^{2}W_{2}}e^{t^{3}W_{3}}\cdots .
\end{equation}
where the $W_k$ can be determined by differential equations step by step, and it is easy to see that
$W_k$ is a polynomial of degree $k$ in the $X_i$. Note that
the multivariable Zassenhaus formula \eqref{x2} is the case when $t=1$.

First we consider the iterated system of equations\\
\begin{align} \label{111}
R_{1}(t)&=e^{-t X_{n}}\cdots e^{-t X_{2}}e^{-t X_{1}}e^{t(X_{1}+X_{2}+\cdots+X_{n})},\\ \label{222}
R_{m}(t)
&=e^{-t^{m}W_{m}}R_{m-1}(t), \qquad m\geq2. 
\end{align}

It follows from \eqref{222} that\\
\begin{equation} \label{333}
R_{m}(t)=e^{t^{m+1}W_{m+1}}e^{t^{m+2}W_{m+2}}\cdots , \qquad m\geq1.
\end{equation}
We then take the logarithmic differentiation\\
\begin{equation} \label{444}
F_{m}(t)=R^{\prime}_{m}(t)R_{m}(t)^{-1}
\qquad m\geq1.
\end{equation}

For $m=1$, we have that 
\begin{align}
F_{1}(t)
&=-X_{n}-e^{-t ad_{X_{n}}}X_{n-1}-e^{-t ad_{X_{n}}}e^{-t ad_{X_{n-1}}}X_{n-2}-\cdots -e^{-t ad_{X_{n}}}\cdots e^{-t ad_{X_{2}}}X_{1} \notag \\
&\quad +e^{-t ad_{X_{n}}}e^{-t ad_{X_{n-1}}}\cdots e^{-t ad_{X_{1}}}(X_{1}+X_{2}+\cdots +X_{n}) \notag \\
&=e^{-t ad_{X_{n}}}(e^{-t ad_{X_{n-1}}}\cdots e^{-t ad_{X_{2}}}e^{-t ad_{X_{1}}}-I)X_{n} \notag \\
&\quad +e^{-t ad_{X_{n}}}e^{-t ad_{X_{n-1}}}(e^{-t ad_{X_{n-2}}}\cdots e^{-t ad_{X_{2}}}e^{-t ad_{X_{1}}}-I)X_{n-1} \notag \\
&\quad +e^{-t ad_{X_{n}}}e^{-t ad_{X_{n-1}}}e^{-t ad_{X_{n-2}}}(e^{-t ad_{X_{n-3}}}\cdots e^{-t ad_{X_{2}}}e^{-t ad_{X_{1}}}-I)X_{n-2}+\cdots \notag \\
&\quad +e^{-t ad_{X_{n}}}e^{-t ad_{X_{n-1}}}\cdots e^{-t ad_{X_{3}}}(e^{-t ad_{X_{2}}}e^{-t ad_{X_{1}}}-I)X_{3} \notag \\
&\quad +e^{-t ad_{X_{n}}}e^{-t ad_{X_{n-1}}}\cdots e^{-t ad_{X_{2}}}(e^{-t ad_{X_{1}}}-I)X_{2} \notag \\
&=\sum_{k=1}^{\infty}(-t)^{k}\sum_{i=2}^{n}\sum_{\substack{j_{1}+\cdots +j_{i-1}\geq 1\\j_{1}+\cdots +j_{n}=k}}\frac{ad_{X_{n}}^{j_{n}}\cdots ad_{X_{2}}^{j_{2}}ad_{X_{1}}^{j_{1}}}{j_{1}!j_{2}!\cdots j_{n}!}X_{i}, \notag\end{align}
where $ad_{A}B=[A, B]$ and we have used the well-known formula
\begin{align*}
e^{A}Be^{-A}=e^{ad_{A}}B=\sum_{n\geq 0}\frac{1}{n!}ad_{A}^{n}B,
\end{align*}
as well as the fact that $e^{ad_X}X=X$.
Write
\begin{align} \label{555}
F_{1}(t)=\sum_{k=1}^{\infty}f_{1, k}t^{k},
\end{align}
then
\begin{align} \label{666}
f_{1, k}=(-1)^{k}\sum_{i=2}^{n}\sum_{\substack{j_{1}+\cdots +j_{i-1}\geq 1\\j_{1}+\cdots +j_{n}=k}}\frac{ad_{X_{n}}^{j_{n}}\cdots ad_{X_{2}}^{j_{2}}ad_{X_{1}}^{j_{1}}}{j_{1}!j_{2}!\cdots j_{n}!}X_{i}.
\end{align}

A similar expansion can be obtained for $F_{m}(t)$, $m\geq 2$, by using $R_{m}(t)$ in \eqref{222}. 
More specifically,
\begin{align*} \label{777}
F_{m}(t)&=-mW_{m}t^{m-1}+e^{-t^{m}W_{m}}R_{m-1}^{\prime}(t)R^{-1}_{m-1}(t)e^{t^{m}W_{m}}\\
&=-mW_{m}t^{m-1}+e^{-t^{m}ad_{W_{m}}}F_{m-1}(t)\\
&=e^{-t^{m}ad_{W_{m}}}(F_{m-1}(t)-mW_{m}t^{m-1}).
\end{align*}

Writing $F_{m}(t)=\sum_{k=m}^{\infty}f_{m, k}t^{k}$, we immediately get that
\begin{align}
f_{m, k}=\sum_{j=o}^{[\frac{k}{m}]-1}\frac{(-1)^{j}}{j!}ad_{W_{m}}^{j}f_{m-1, k-mj}, \qquad k\geq m
\end{align}
where $[\frac{k}{m}]$ denotes the integer part of $\frac{k}{m}$.

%
On the other hand, if we take the logarithmic derivative of $R_m(t)$ using the expression \eqref{333}, we arrive at
\begin{align} \label{999}
F_{m}(t)=(m+1)W_{m+1}t^{m}+\sum_{j=m+2}^{\infty}jt^{j-1}e^{t^{m+1}ad_{W_{m+1}}}\cdots e^{t^{j-1}ad_{W_{j-1}}}W_{j}.
\end{align}

Comparing the coefficients of the terms $t$, $t^{2}$, $t^{3}$ and $t^4$ in \eqref{555} and \eqref{999} for $F_1(t)$, we get that
\begin{align*}
f_{1, 1}=2W_{2}, \quad
f_{1, 2}=3W_{3}, \quad
f_{1, 3}=4W_{4}, \quad
f_{1, 4}=5W_5+3[W_2, W_3],
\end{align*}
so that
\begin{align}\label{les5}
W_{2}=\frac{1}{2}f_{1, 1}, \quad
W_{3}=\frac{1}{3}f_{1, 2}, \quad
W_{4}=\frac{1}{4}f_{1,3}, \quad
W_5=\frac15 f_{1, 4}-\frac1{10}[f_{1, 1}, f_{1, 2}].
\end{align}

Similarly, comparing \eqref{777} and \eqref{999}, we get
$$f_{m, m}=(m+1)W_{m+1},$$
therefore
$$W_{m+1}=\frac{1}{m+1}f_{m, m}=\frac{1}{m+1}f_{[\frac{m}{2}],m}, \quad m\geq 4, $$
i.e. $$\quad W_{m}=\frac{1}{m}f_{[\frac{m-1}{2}], m-1}, \quad m\geq5$$

\subsection{Examples of $W_k$.}
When $k=1$ in the expression \eqref{666},  the summation of the first $i-1$ terms is already at least $1$, so we have the formula 
\begin{align*}
f_{1, 1}&=-\sum_{i=2}^{n}\sum_{\substack{j_{1}+\cdots +j_{i-1}=1}}\frac{ad_{X_{i-1}}^{j_{i-1}}\cdots ad_{X_{2}}^{j_{2}}ad_{X_{1}}^{j_{1}}}{j_{1}!j_{2}!\cdots j_{i-1}!}X_{i}\\
&=\sum_{1\leq i<j\leq n}[X_{j},X_{i}].
\end{align*}

Thus
\begin{align*}
W_{2}=\frac{1}{2}f_{1, 1}=\frac{1}{2}\sum_{1\leq i<j\leq n}[X_{j},X_{i}].
\end{align*}
Similarly for $k=2$ in \eqref{666}, we have
\begin{align*}
f_{1, 2}&=\sum_{i=2}^{n}\sum_{\substack{j_{1}+\cdots +j_{i-1}\geq 1\\j_{1}+\cdots +j_{n}=2}}\frac{ad_{X_{n}}^{j_{n}}\cdots ad_{X_{2}}^{j_{2}}ad_{X_{1}}^{j_{1}}}{j_{1}!j_{2}!\cdots j_{n}!}X_{i}\\
&=\sum_{i=2}^{n}\left(\sum_{\substack{j_{1}+\cdots +j_{i-1}=1\\i\leq l\leq n}}\frac{ad_{X_l}ad_{X_{i-1}}^{j_{i-1}}\cdots ad_{X_{1}}^{j_{1}}}{j_{1}!j_{2}!\cdots j_{i-1}!}X_{i}+
\sum_{\substack{j_{1}+\cdots +j_{i-1}=2}}\frac{ad_{X_{i-1}}^{j_{i-1}}\cdots ad_{X_{1}}^{j_{1}}}{j_{1}!j_{2}!\cdots j_{i-1}!}X_{i}\right)\\
&=\sum_{i=2}^{n}\left(\sum_{\substack{1\leq j\leq i-1\\i\leq l\leq n}}{ad_{X_l}ad_{X_{j}}}X_{i}+
\sum_{\substack{1\leq j_1\leq j_2\leq i-1\\m_1+m_2=2}}\frac{ad_{X_{j_2}}ad_{X_{j_1}}}{m_1!m_2!}X_{i}\right)\\
&=-\sum_{l\geq i>j}[X_l[X_i, X_j]]+\sum_{i>j_2\geq j_1}\frac1{m_j!}([[X_i, X_{j_2}]X_{j_1}]-[X_i, [X_{j_2}, X_{j_1}]])\\
&=-\sum_{i>j}[X_i[X_iX_j]]-2\sum_{i>j>k}[X_i[X_j X_k]]+\sum_{i>j>k}[[X_iX_{j}]X_{k}]+\sum_{i>j}\frac12[[X_iX_j]X_j]\\
&=\sum_{1\leq i<j, k\leq n}[[X_{j}, X_{i}],X_{k}]+\frac{1}{2}\sum_{1\leq i<j\leq n}[[X_{j}, X_{i}], X_{i}],
\end{align*}
where $m_i$ is the multiplicity of $j_i$.

Therefore
$$
W_{3}=\frac{1}{3}f_{1, 2}=\frac{1}{3}\sum_{1\leq i<j, k\leq n}[[X_{j}, X_{i}],X_{k}]+\frac{1}{6}\sum_{1\leq i<j\leq n}[[X_{j}, X_{i}], X_{i}].
$$

We list the first few other terms as follows:
\begin{align*}
f_{1,3}&=-(\sum_{\substack{1\leq i<j, k\leq n \\ k<l\leq n}}[X_{l}, [X_{k},[X_{i}, X_{j}]]]+\frac{1}{2}\sum_{1\leq i<j, k\leq n}[X_{k}, [X_{k},[X_{i}, X_{j}]]]\\
& \quad +\frac{1}{2}\sum_{1\leq i<j, k\leq n}[X_{k}, [X_{i},[X_{i}, X_{j}]]]+\frac{1}{6}\sum_{1\leq i<j\leq n}[X_{i}, [X_{i},[X_{i}, X_{j}]]])\\
&=\frac{1}{6}\sum_{1\leq i<j\leq n}[[[X_{j}, X_{i}], X_{i}], X_{i}]+\frac{1}{2}\sum_{1\leq i<j, k\leq n}([[[X_{j}, X_{i}], X_{i}], X_{k}]\\
& \quad +[[[X_{j}, X_{i}], X_{k}], X_{k}])+\sum_{\substack{1\leq i<j, k\leq n\\k<l\leq n}}[[[X_{j}, X_{i}], X_{k}], X_{l}],
\end{align*}
so that
\begin{align*}
W_{4}
&=\frac{1}{24}\sum_{1\leq i<j\leq n}[[[X_{j}, X_{i}], X_{i}], X_{i}]+\frac{1}{8}\sum_{1\leq i<j, k\leq n}([[[X_{j}, X_{i}], X_{i}], X_{k}]\\
& \quad +[[[X_{j}, X_{i}], X_{k}], X_{k}]) +\frac{1}{4}\sum_{\substack{1\leq i<j, k\leq n\\k<l\leq n}}[[[X_{j}, X_{i}], X_{k}], X_{l}].
\end{align*}
\begin{align*}
f_{1,4}&=\frac{1}{24}\sum_{1\leq i<j\leq n}[[[[X_{j}, X_{i}], X_{i}], X_{i}], X_{i}]+\frac{1}{6}\sum_{1\leq i<j, k\leq n}([[[[X_{j}, X_{i}], X_{i}], X_{i}], X_{k}]\\
& \quad +[[[[X_{j}, X_{i}], X_{k}], X_{k}], X_{k}])+\frac{1}{4}\sum_{1\leq i<j, k\leq n}[[[[X_{j}, X_{i}], X_{i}]X_{k}]X_{k}]\\
& \quad +\frac{1}{2}\sum_{\substack{1\leq i<j, k\leq n\\k<l\leq n}}([[[[X_{j}, X_{i}], X_{k}], X_{l}], X_{l}]+[[[[X_{j}, X_{i}], X_{k}], X_{k}], X_{l}]\\
& \quad +[[[[X_{j}, X_{i}], X_{i}], X_{k}], X_{l}])+\sum_{\substack{1\leq i<j, k\leq n\\k<l<h\leq n}}[[[[X_{j}, X_{i}], X_{k}], X_{l}], X_{h}].
\end{align*}
\begin{align*}
f_{1, 5}&=\frac{1}{120}\sum_{1\leq i<j\leq n}[[[[[X_{j}, X_{i}], X_{i}], X_{i}], X_{i}], X_{i}]+\frac{1}{24}\sum_{1\leq i<j, k\leq n}([[[[[X_{j}, X_{i}], X_{i}], X_{i}], X_{i}], X_{k}]\\
& \quad +[[[[[X_{j}, X_{i}], X_{k}], X_{k}], X_{k}], X_{k}])+\frac{1}{12}\sum_{1\leq i<j, k\leq n}([[[[[X_{j}, X_{i}], X_{i}], X_{i}],X_{k}],X_{k}]\\
& \quad +[[[[[X_{j}, X_{i}], X_{i}], X_{k}], X_{k}], X_{k}])+\frac{1}{6}\sum_{\substack{1\leq i<j, k\leq n\\k<l\leq n}}([[[[[X_{j}, X_{i}], X_{i}], X_{i}],X_{k}],X_{l}]\\
& \quad +[[[[[X_{j}, X_{i}], X_{k}], X_{k}],X_{k}],X_{l}]+[[[[[X_{j}, X_{i}], X_{k}], X_{l}],X_{l}],X_{l}])\\
& \quad +\frac{1}{4}\sum_{\substack{1\leq i<j, k\leq n\\k<l\leq n}}([[[[[X_{j}, X_{i}], X_{i}], X_{k}],X_{k}],X_{l}]+[[[[[X_{j}, X_{i}], X_{k}], X_{k}],X_{l}],X_{l}]\\
& \quad +[[[[[X_{j}, X_{i}], X_{i}], X_{k}],X_{l}],X_{l}])+\frac{1}{2}\sum_{\substack{1\leq i<j, k\leq n\\k<l<h\leq n}}([[[[[X_{j}, X_{i}], X_{i}], X_{k}],X_{l}],X_{h}]\\
& \quad +[[[[[X_{j}, X_{i}], X_{k}], X_{k}],X_{l}],X_{h}]+[[[[[X_{j}, X_{i}], X_{k}], X_{l}],X_{l}],X_{h}]\\
& \quad+[[[[[X_{j}, X_{i}], X_{k}], X_{l}],X_{h}],X_{h}])+\sum_{\substack{1\leq i<j, k\leq n\\k<l<h<m\leq n}}[[[[[X_{j}, X_{i}], X_{k}], X_{l}],X_{h}],X_{m}].
\end{align*}

\section{Iteration Formulas}

To reveal the explicit rule for $f_{1, k} (k\geq 1)$ based on the computations we gave in Section 2,
we recall the definition of partitions and compositions \cite{Andrews}.

\subsection{Formulation in terms of partitions.}~
A partition of a positive integer $m$ is an integral unordered decomposition $m=\lambda_1+\cdots +\lambda_l$ such that $\lambda_1\geq\cdots \geq \lambda_l>0$, denoted by $\lambda=(\lambda_1\lambda_2\dots \lambda_l)$ and $\lambda\vdash m$. Here $\lambda_i$ are called the parts and $l$ is the length of the partition.
A composition is an ordered integral decomposition of $m$: $m=\lambda_1+\cdots +\lambda_l$ such that $\lambda
_i>0$ and denoted by $\lambda \vDash m$, in another words,
compositions of $m$ are obtained by permuting the unequal parts of the associated partition of $m$. The set of partitions of
$m$ is denoted by $\mathcal{P}(m)$ and the cardinality is denoted by $p(m)$.

For example, the partitions of $4$ are:
\begin{align*}
(\lambda)^{1}&=(4), \quad &(\lambda)^{2}&=(3, 1), \quad &(\lambda)^{3}&=(2, 2),\\
(\lambda)^{4}&=(2, 1, 1), \quad &(\lambda)^{5}&=(1, 1, 1, 1).
 \end{align*}
Therefore, $p(4)=5$. The associated compositions are distinct permutations of the partitions:
$(4), (3, 1), (1, 3), (2, 2), (2, 1, 1), (1, 2, 1), (1, 1, 2), (1, 1, 1, 1)$.

Accordingly the formulas of $f_{1, k}$ go as follows.
For $f_{1, 1}$, $p(1)=1$,
$$
(\lambda)^{1}=(1): \quad \sum_{1\leq i<j\leq n}[X_{j}, X_{i}].
$$
For $f_{1, 2}$, $p(2)=2$,
\begin{align*}
&(\lambda)^{1}=(2): \quad \frac{1}{2!}\sum_{1\leq i<j\leq n}[[X_{j},X_{i}], X_{i}], \\
&(\lambda)^{2}=(1, 1): \quad \sum_{1\leq i<j, k\leq n}[[X_{j},X_{i}], X_{k}].
\end{align*}
For $f_{1, 3}$, $p(3)=3$,
\begin{align*}
&(\lambda)^{1}=(3): \quad \frac{1}{3!}\sum_{1\leq i<j\leq n}[[[X_{j}, X_{i}], X_{i}], X_{i}], \\
&(\lambda)^{2}=(2, 1): \quad \frac{1}{2!}\sum_{1\leq i<j, k\leq n}([[[X_{j}, X_{i}], X_{i}], X_{k}]+[[[X_{j}, X_{i}], X_{k}], X_{k}]), \\
&(\lambda)^{3}=(1, 1, 1): \quad \sum_{\substack{1\leq i<j, k\leq n\\k<l\leq n}}[[[X_{j}, X_{i}], X_{k}], X_{l}].
\end{align*}
For $f_{1, 4}$, $p(4)=5$,
\begin{align*}
&(\lambda)^{1}=(4): \quad \frac{1}{4!}\sum_{1\leq i<j\leq n}[[[[X_{j}, X_{i}], X_{i}], X_{i}], X_{i}], \\
&(\lambda)^{2}=(3, 1): \quad \frac{1}{3!}\sum_{1\leq i<j, k\leq n}([[[[X_{j}, X_{i}], X_{i}], X_{i}], X_{k}]+[[[[X_{j}, X_{i}], X_{k}], X_{k}], X_{k}]), \\
&(\lambda)^{3}=(2, 2): \quad \frac{1}{2!2!}\sum_{1\leq i<j, k\leq n}[[[[X_{j}, X_{i}], X_{i}], X_{k}], X_{k}], \\
&(\lambda)^{4}=(2, 1, 1): \quad \frac{1}{2!}\sum_{\substack{1\leq i<j, k\leq n\\k<l\leq n}}([[[[X_{j}, X_{i}], X_{i}], X_{k}], X_{l}]+[[[[X_{j}, X_{i}], X_{k}], X_{k}], X_{l}]\\
&\qquad \qquad \qquad \qquad \qquad \qquad +[[[[X_{j}, X_{i}], X_{k}], X_{l}], X_{l}]), \\
&(\lambda)^{5}=(1, 1 , 1, 1): \quad \sum_{\substack{1\leq i<j, k\leq n\\k<l<h\leq n}}[[[[X_{j}, X_{i}], X_{k}], X_{l}], X_{h}].
\end{align*}

We define the long commutator inductively as follows.
\begin{align*}
[X_1, X_2]&=X_1X_2-X_2X_1,\\
[X_1, X_2, X_3]&=[[X_1, X_2], X_3],\\
[X_1, X_2, X_3, \cdots, X_i]&=[[X_1, X_2], X_3, \cdots, X_i].
\end{align*}

Fix a partition $\lambda=(\lambda_1\lambda_2\cdots \lambda_l)$ of $k$, and for each composition out of $\lambda$: $(k_1k_2\cdots k_l)\models k$ which
is a rearrangement of $\lambda$ by permuting its parts, we associate the commutator
\begin{equation}\label{e:multcomm}
[X_j, X_{i_1}, \cdots, X_{i_1}, X_{i_2}, \cdots, X_{i_2}, \cdots, X_{i_l}, \cdots, X_{i_l}]
\end{equation}
where the multiplicity of $i_s$ is $k_s$ for $1\leq s\leq l$. For this reason, we will write \eqref{e:multcomm} as $[X_jX_{i_1}^{k_1}X_{i_2}^{k_2}\cdots X_{i_l}^{k_l}]$. Then we have the following result.

\begin{theorem}
\label{3.1} For each $k$, the following formula holds
\begin{align*}
f_{1, k}=\sum_{(k_1\cdots k_l)\models k}\frac{1}{k_{1}!k_{2}!\cdots k_{l}!}\sum_{\substack{1\leq i_{1}<j, i_{2}\leq n\\i_{2}<i_{3}<\cdots <i_{l}\leq n}}[X_jX_{i_1}^{k_1}X_{i_2}^{k_2}\cdots X_{i_l}^{k_l}]
\end{align*}
\end{theorem}

\subsection{Determination of $W_{m}(m\geq 6)$.}~
We have given the formulas of $W_k$ for $1\leq k\leq 5$ in terms of $f_{1, k}$ \eqref{les5}. We now give the next a few terms as follows.

\begin{align}
\quad W_{6}&=\frac{1}{6}(f_{1, 5}-ad_{W_{2}}f_{1, 3}). \label{2}\\
\quad W_{7}&=\frac{1}{7}(f_{1, 6}-ad_{W_{2}}f_{1, 4}+\frac{1}{2!}ad_{W_{2}}^{2}f_{1, 2}-ad_{W_{3}}f_{1, 3}). \label{3}\\
\quad W_{8}&=\frac{1}{8}(f_{1, 7}-ad_{W_{2}}f_{1, 5}+\frac{1}{2!}ad_{W_{2}}^{2}f_{1, 3}-ad_{W_{3}}f_{1, 4}+ad_{W_{3}}ad_{W_{2}}f_{1, 2}). \label{4}\\
\quad W_{9}&=\frac{1}{9}(f_{1, 8}-ad_{W_{2}}f_{1, 6}+\frac{1}{2!}ad_{W_{2}}^{2}f_{1, 4}-\frac{1}{3!}ad_{W_{2}}^{3}f_{1, 2}-ad_{W_{3}}f_{1, 5} \label{5}\\
&\qquad +ad_{W_{3}}ad_{W_{2}}f_{1, 3}-ad_{W_{4}}f_{1, 4}+ad_{W_{4}}ad_{W_{2}}f_{1, 2}). \notag\\
\label{6}
W_{10}&=\frac{1}{10}(f_{1, 9}-ad_{W_{2}}f_{1, 7}+\frac{1}{2!}ad_{W_{2}}^{2}f_{1, 5}-\frac{1}{3!}ad_{W_{2}}^{3}f_{1, 3}-ad_{W_{3}}f_{1, 6}\\
&\qquad +ad_{W_{3}}ad_{W_{2}}f_{1, 4}-\frac{1}{2!}ad_{W_{3}}ad_{W_{2}}^{2}f_{1, 2}+\frac{1}{2!}ad_{W_{3}}^{2}f_{1, 3} \notag\\
&\qquad -ad_{W_{4}}f_{1, 5}+ad_{W_{4}}ad_{W_{2}}f_{1, 3}). \notag
\end{align}

We postpone the verification of these formulas till the general result.
The following result gives the general iterative formula for the multivariable Zassenhaus formula.
\begin{theorem} \label{3.2} For each $k\geq 2$ the exponents $W_{m}$ in the multi-variable Zassenhaus formula \eqref{x2}
for $m=6k+i$, where $i=0, 1, 2, 3, 4, 5$ are given by\\
\begin{align}
W_{6k}=&\frac{1}{6k}(f_{2k-2, 6k-1}-ad_{W_{2k-1}}f_{2k-2, 4k}+\frac{1}{2!}ad_{W_{2k-1}}^{2}f_{2k-2, 2k+1}-ad_{W_{2k}}f_{2k-2, 4k-1} \notag\\
&+ad_{W_{2k}}ad_{W_{2k-1}}f_{2k-2, 2k}-ad_{W_{2k+1}}f_{2k-2, 4k-2}+ad_{W_{2k+1}}ad_{W_{2k-1}}f_{2k-2, 2k-1} \notag\\
&-ad_{W_{2k+2}}f_{2k-2, 4k-3}-ad_{W_{2k+3}}f_{2k-2, 4k-4}-\cdots -ad_{W_{3k-1}}f_{2k-2, 3k}). \label{7}\\
W_{6k+1}=&\frac{1}{6k+1}(f_{2k-1, 6k}-ad_{W_{2k}}f_{2k-1, 4k}+\frac{1}{2!}ad_{W_{2k}}^{2}f_{2k-1, 2k}-ad_{W_{2k+1}}f_{2k-1, 4k-1} \qquad \notag\\
&-ad_{W_{2k+2}}f_{2k-1, 4k-2} -ad_{W_{2k+3}}f_{2k-1, 4k-3}-\cdots -ad_{W_{3k}}f_{2k-1, 3k}).  \label{8}\\ 
W_{6k+2}=&\frac{1}{6k+2}(f_{2k-1, 6k+1}-ad_{W_{2k}}f_{2k-1, 4k+1}+\frac{1}{2!}ad_{W_{2k}}^{2}f_{2k-1, 2k+1}-ad_{W_{2k+1}}f_{2k-1, 4k} \notag\\
&+ad_{W_{2k+1}}ad_{W_{2k}}f_{2k-1, 2k}-ad_{W_{2k+2}}f_{2k-1, 4k-1}-ad_{W_{2k+3}}f_{2k-1, 4k-2} \notag\\
&-ad_{W_{2k+4}}f_{2k-1, 4k-3}-\cdots -ad_{W_{3k}}f_{2k-1, 3k+1}). \label{9}\\ 
W_{6k+3}=&\frac{1}{6k+3}(f_{2k-1, 6k+2}-ad_{W_{2k}}f_{2k-1, 4k+2}+\frac{1}{2!}ad_{W_{2k}}^{2}f_{2k-1, 2k+2} \notag\\
&-ad_{W_{2k+1}}f_{2k-1, 4k+1}+ad_{W_{2k+1}}ad_{W_{2k}}f_{2k-1, 2k+1}-ad_{W_{2k+2}}f_{2k-1, 4k} \notag\\
&+ad_{W_{2k+2}}ad_{W_{2k}}f_{2k-1, 2k}-ad_{W_{2k+3}}f_{2k-1, 4k-1}-ad_{W_{2k+4}}f_{2k-1, 4k-2} \notag\\
&-\cdots -ad_{W_{3k+1}}f_{2k-1, 3k+1}). \label{10}
\end{align}
\begin{align}
W_{6k+4}=&\frac{1}{6k+4}(f_{2k, 6k+3}-ad_{W_{2k+1}}f_{2k, 4k+2}+\frac{1}{2!}ad_{W_{2k+1}}^{2}f_{2k, 2k+1}-ad_{W_{2k+2}}f_{2k, 4k+1} \notag\\
&-ad_{W_{2k+3}}f_{2k, 4k} -ad_{W_{2k+4}}f_{2k, 4k-1}-\cdots -ad_{W_{3k+1}}f_{2k, 3k+2}). \label{11}\\
W_{6k+5}=&\frac{1}{6k+5}(f_{2k, 6k+4}-ad_{W_{2k+1}}f_{2k, 4k+3}+\frac{1}{2!}ad_{W_{2k+1}}^{2}f_{2k, 2k+2}-ad_{W_{2k+2}}f_{2k, 4k+2} \notag\\
&+ad_{W_{2k+2}}ad_{W_{2k+1}}f_{2k, 2k+1}-ad_{W_{2k+3}}f_{2k, 4k+1}-ad_{W_{2k+4}}f_{2k, 4k} \notag\\
&-ad_{W_{2k+5}}f_{2k, 4k-1}-\cdots -ad_{W_{3k+2}}f_{2k, 3k+2}). \label{12}
\end{align}
\end{theorem}
\begin{proof} As we know that
$W_{m}=\frac{1}{m}f_{[\frac{m-1}{2}], m-1}$, $m\geq 5$ in Section 2, we divide $m$ into even and odd integers.

When $m=2a+1, a\geq 2$,
\begin{align*}
f_{[\frac{m-1}{2}], m-1}&=f_{a, m-1}\qquad ([\frac{m-1}{a}]=2)\\
&=f_{a-1, m-1}-ad_{W_{a}}f_{a-1, a},
\end{align*}
if $a=2$, we stop the computation since we reach $f_{1, k}$.
Otherwise,
\begin{equation}
[\frac{m-1}{a-1}]=[2+\frac{2}{a-1}]=\begin{cases}
3, &  a=3; \\
2, &  a\geq 4.
\end{cases} \notag
\end{equation}
$$
[\frac{a}{a-1}]=[1+\frac{1}{a-1}]=1, \quad a\geq 3.
$$
so that
if $a=3$,
$$
f_{[\frac{m-1}{2}], m-1}=f_{a-2, m-1}-ad_{W_{a-1}}f_{a-2,a+1}+\frac{1}{2!}ad_{W_{a-1}}^{2}f_{a-2, 2}-ad_{W_{a}}f_{a-2,a},
$$
we stop the computation.
If $a\geq 4$,
$$
f_{[\frac{m-1}{2}], m-1}=f_{a-2, m-1}-ad_{W_{a-1}}f_{a-2, a+1}-ad_{W_{a}}f_{a-2, a}.
$$
Repeating the procedure, we obtain 
\eqref{3} and \eqref{5} as well as \eqref{8}, \eqref{10}, \eqref{12} in the theorem by using induction. 

Similarly, when $m=2a, a\geq 3$,
\begin{align*}
f_{[\frac{m-1}{2}], m-1}&=f_{a-1, m-1}\qquad ([\frac{m-1}{a-1}]=[2+\frac{1}{a-1}]=2, a\geq 3)\\
&=f_{a-2, m-1}-ad_{W_{a-1}}f_{a-2, a},
\end{align*}
if $a=3$, we stop the computation.
Otherwise,
\begin{equation}
[\frac{m-1}{a-2}]=[2+\frac{3}{a-2}]=\begin{cases}
3, &  a=4; \\
3, &  a=5; \\
2, &  a\geq 6.
\end{cases} \notag
\end{equation}
\begin{equation}
[\frac{a}{a-2}]=[1+\frac{2}{a-2}]=\begin{cases}
2, &  a=4; \\
1, &  a\geq 5.
\end{cases} \notag
\end{equation}
so that
if $a=4$,
\begin{align*}
f_{[\frac{m-1}{2}], m-1}&=f_{a-3,m-1}-ad_{W_{a-2}}f_{a-3, a+1}+\frac{1}{2!}ad_{W_{a-2}}^{2}f_{a-3, 3}\\
& \quad -ad_{W_{a-1}}f_{a-3, a}+ad_{W_{a-1}}ad_{W_{a-2}}f_{a-3, 2},
\end{align*}
we stop the computation.
If $a=5$,
$$
f_{[\frac{m-1}{2}], m-1}=f_{a-3,m-1}-ad_{W_{a-2}}f_{a-3, a+1}+\frac{1}{2!}ad_{W_{a-2}}^{2}f_{a-3, 3}-ad_{W_{a-1}}f_{a-3, a}.
$$
If $a\geq 6$
$$
f_{[\frac{m-1}{2}], m-1}=f_{a-3,m-1}-ad_{W_{a-2}}f_{a-3, a+1}-ad_{W_{a-1}}f_{a-3, a}.
$$
Repeating the procedure, we obtain \eqref{2}, \eqref{4} and \eqref{6} as well as
\eqref{7}, \eqref{9}, \eqref{11} in Theorem \ref{3.2} using induction. 
\end{proof}

According to Theorem \ref{3.2}, we know that $W_{m}(m\geq 5)$ can be expressed as a linear combination of $f_{1, k}(k\geq 1)$ in the end, then we use $f_{1, k}(k\geq 1)$ given in Theorem \ref{3.1} to obtain $W_{m}(m\geq 5)$. To explain how this works, we
give the explicit formulas of $W_{5}, W_{6}$ according to \eqref{les5} and \eqref{2}:
\begin{small}
\begin{align*}
W_{5}=&\frac{1}{5}(f_{1, 4}-ad_{W_{2}}f_{1, 2})\\
=&\frac{1}{120}\sum_{1\leq i<j\leq n}[[[[X_{j}, X_{i}], X_{i}], X_{i}], X_{i}]+\frac{1}{30}\sum_{1\leq i<j, k\leq n}([[[[X_{j}, X_{i}], X_{i}], X_{i}], X_{k}]\\
&+[[[[X_{j}, X_{i}], X_{k}], X_{k}], X_{k}])+\frac{1}{20}\sum_{1\leq i<j, k\leq n}[[[[X_{j}, X_{i}], X_{i}]X_{k}]X_{k}]\\
&+\frac{1}{10}\sum_{\substack{1\leq i<j, k\leq n\\k<l\leq n}}([[[[X_{j}, X_{i}], X_{k}], X_{l}], X_{l}]+[[[[X_{j}, X_{i}], X_{k}], X_{k}], X_{l}]\\
& +[[[[X_{j}, X_{i}], X_{i}], X_{k}], X_{l}])+\frac{1}{5}\sum_{\substack{1\leq i<j, k\leq n\\k<l<h\leq n}}[[[[X_{j}, X_{i}], X_{k}], X_{l}], X_{h}]\\
&+\frac{1}{10}\sum_{\substack{1\leq i_{1}<j_{1}\leq n\\1\leq i_{2}<j_{2}, k_{2}\leq n}}[[[X_{j_{2}}, X_{i_{2}}], X_{k_{2}}], [X_{j_{1}}, X_{i_{1}}]]+\frac{1}{20}\sum_{\substack{1\leq i_{1}<j_{1}\leq n\\1\leq i_{3}<j_{3}\leq n}}[[[X_{j_{3}}, X_{i_{3}}], X_{i_{3}}], [X_{j_{1}}, X_{i_{1}}]].
\end{align*}
\begin{align*}
W_{6}=&\frac{1}{6}(f_{1, 5}-ad_{W_{2}}f_{1, 3})\\
=&\frac{1}{720}\sum_{1\leq i<j\leq n}[[[[[X_{j}, X_{i}], X_{i}], X_{i}], X_{i}], X_{i}]+\frac{1}{144}\sum_{1\leq i<j, k\leq n}([[[[[X_{j}, X_{i}], X_{i}], X_{i}], X_{i}], X_{k}]\\
& +[[[[[X_{j}, X_{i}], X_{k}], X_{k}], X_{k}], X_{k}])+\frac{1}{72}\sum_{1\leq i<j, k\leq n}([[[[[X_{j}, X_{i}], X_{i}], X_{i}],X_{k}],X_{k}]\\
&+[[[[[X_{j}, X_{i}], X_{i}], X_{k}], X_{k}], X_{k}])+\frac{1}{36}\sum_{\substack{1\leq i<j, k\leq n\\k<l\leq n}}([[[[[X_{j}, X_{i}], X_{i}], X_{i}],X_{k}],X_{l}]\\
& +[[[[[X_{j}, X_{i}], X_{k}], X_{k}],X_{k}],X_{l}]+[[[[[X_{j}, X_{i}], X_{k}], X_{l}],X_{l}],X_{l}])\\
& +\frac{1}{24}\sum_{\substack{1\leq i<j, k\leq n\\k<l\leq n}}([[[[[X_{j}, X_{i}], X_{i}], X_{k}],X_{k}],X_{l}]+[[[[[X_{j}, X_{i}], X_{k}], X_{k}],X_{l}],X_{l}]\\
&  +[[[[[X_{j}, X_{i}], X_{i}], X_{k}],X_{l}],X_{l}])+\frac{1}{12}\sum_{\substack{1\leq i<j, k\leq n\\k<l<h\leq n}}([[[[[X_{j}, X_{i}], X_{i}], X_{k}],X_{l}],X_{h}]\\
&  +[[[[[X_{j}, X_{i}], X_{k}], X_{k}],X_{l}],X_{h}]+[[[[[X_{j}, X_{i}], X_{k}], X_{l}],X_{l}],X_{h}]\\
& +[[[[[X_{j}, X_{i}], X_{k}], X_{l}],X_{h}],X_{h}])+\frac{1}{6}\sum_{\substack{1\leq i<j, k\leq n\\k<l<h<m\leq n}}[[[[[X_{j}, X_{i}], X_{k}], X_{l}],X_{h}],X_{m}]\\
& +\frac{1}{72}\sum_{\substack{1\leq i_{1}<j_{1}\leq n\\1\leq i_{2}<j_{2}\leq n}}[[[[X_{j_{2}}, X_{i_{2}}], X_{i_{2}}],  X_{i_{2}}], [X_{j_{1}}, X_{i_{1}}]]\\
&  +\frac{1}{24}\sum_{\substack{1\leq i_{1}<j_{1}\leq n\\1\leq i_{3}<j_{3}, k_{3}\leq n}} ([[[[X_{j_{3}}, X_{i_{3}}], X_{i_{3}}],  X_{k_{3}}], [X_{j_{1}}, X_{i_{1}}]]\\
& +[[[[X_{j_{3}}, X_{i_{3}}], X_{k_{3}}],  X_{k_{3}}], [X_{j_{1}}, X_{i_{1}}]])+\frac{1}{12}\sum_{\substack{1\leq i_{1}<j_{1}\leq n\\1\leq i_{4}<j_{4}, k_{4}\leq n\\k_{4}<l_{4}\leq n}}[[[[X_{j_{4}}, X_{i_{4}}], X_{k_{4}}],  X_{l_{4}}], [X_{j_{1}}, X_{i_{1}}]].
\end{align*}
\end{small}

\centerline{\bf Acknowledgments}

{N. Jing's work was partially supported by the National Natural Science Foundation of China (Grant No.11531004) and
Simons Foundation (Grant No. 523868).}

\end{document}